\documentclass[reqno]{amsart}

\usepackage[paper=letterpaper,margin=1in,twoside=false,includehead,headheight=18pt]{geometry}

\usepackage[T3,T1]{fontenc}

\usepackage{dsfont,tikz}
\usepackage{amsmath,amssymb,amsthm} 
\usepackage{enumerate,paralist}
\usepackage{graphicx}
\usepackage{mathtools}

\usepackage{todonotes}

\usetikzlibrary{decorations.pathreplacing}

\newtheorem{theorem}{Theorem}
\newtheorem{lemma}[theorem]{Lemma} 
 
\newtheorem{corollary}[theorem]{Corollary}

{\theoremstyle{definition} 

\newtheorem*{definition}{Definition} 
\newtheorem*{conjecture}{Conjecture} 
 
 }

\newtheoremstyle{named}%
  {}{}						
  {\upshape}				
  {0pt}{\bfseries}			
  {.}						
  {.5em}					
  {\thmname{#1}\thmnote{ #3}}  

\theoremstyle{named}

\newcommand{\wo}{\setminus}

\newcommand{\set}[1]{\left\{{#1}\right\}} 
\newcommand{\setof}[2]{\left\{{#1}\,:\,{#2}\right\}}
\newcommand{\of}{\subseteq}

\newcommand{\0}{\emptyset}
\newcommand{\1}{\mathds{1}}

\newcommand{\given}{\mathop{|}}
\newcommand{\R}{\mathbb{R}}

\renewcommand{\P}{\mathbb{P}}
\newcommand{\E}{\mathbb{E}}

\renewcommand{\H}{\mathcal{H}}

\renewcommand{\l}{\lambda}
\renewcommand{\L}{\Lambda}
\renewcommand{\a}{\alpha}
\newcommand{\denom}{1+6\l+6\l^2}

\DeclareMathOperator{\Ind}{Ind}
\DeclareMathOperator{\ind}{ind}
\DeclareMathOperator{\LP}{LP}
\DeclareMathOperator{\LPP}{LP^{+}}
\DeclareMathOperator{\HC}{HC}

\DeclareMathOperator{\range}{range}
\DeclareMathOperator{\GP}{GP}

\DeclarePairedDelimiter{\abs}{\lvert}{\rvert}

\begin{document}
\title{Minimizing the number of independent sets in triangle-free regular graphs}
\author{Jonathan Cutler} \address{Department of Mathematical Sciences\\
Montclair State University\\
Montclair, NJ} \email{jonathan.cutler@montclair.edu} 
\thanks{The first author was sponsored by the National Security Agency under Grant H98230-15-1-0016.  The United States Government is authorized to reproduce and distribute reprints notwithstanding any copyright notation herein.  The second author was sponsored by the Simons Foundation under Grant 429383.}
\author{A.~J.~Radcliffe} \address{Department of Mathematics\\
University of Nebraska-Lincoln\\
Lincoln, NE} \email{jamie.radcliffe@unl.edu}
\date{\today}
\begin{abstract}
	Recently, Davies, Jenssen, Perkins, and Roberts gave a very nice proof of the result (due, in various parts, to Kahn, Galvin-Tetali, and Zhao) that the independence polynomial of a $d$-regular graph is maximized by disjoint copies of $K_{d,d}$.  Their proof uses linear programming bounds on the distribution of a cleverly chosen random variable.  In this paper, we use this method to give lower bounds on the independence polynomial of regular graphs.  We also give new bounds on the number of independent sets in triangle-free regular graphs.
\end{abstract}
\maketitle

\section{Introduction} 
\label{sec:intro}

Extremal problems involving the number of substructures of a graph of a given type have popped up in quite a few different contexts of late.  One of the best known such results is due to Kahn \cite{K01} and Zhao \cite{Z}.  We let $\Ind(G)$ be the set of independent sets in $G$.  Their theorem bounds $\ind(G)=\abs{\Ind(G)}$ for regular graphs.

\begin{theorem}[Kahn, Zhao]\label{thm:kz}
	If $G$ is a $d$-regular graph on $n$ vertices, then 
	\[
		\ind(G)^{1/n}\leq \ind(K_{d,d})^{1/2d}.
	\]
\end{theorem}

One source for questions of this type is the field of statistical mechanics.  For instance, the \emph{hard-core model} on a graph $G$ is a probability distribution on the independent sets of $G$ in which a independent set $I$ is chosen with probability proportional to $\l^{\abs{I}}$.  Here $\l>0$ is a parameter called the \emph{fugacity}.  The normalizing factor is
	\[
		P_G(\l)=\sum_{I\in \Ind(G)} \l^{\abs{I}},
	\]
known to graph theorists as the \emph{independence polynomial of $G$} and to statistical physicists as the \emph{partition function} of this hard-core model.

Kahn \cite{K01} in fact proved the analogue of Theorem~\ref{thm:kz} for the independence polynomial of bipartite graphs with fugacity $\l\geq 1$, i.e., 
\[
	P_G(\l)^{1/n}\leq P_{K_{d,d}}(\l)^{1/2d}.
\]  
Galvin and Tetali \cite{GT} extended Kahn's result to cover the case $0<\l<1$.  Finally, Zhao \cite{Z} proved the full theorem using a clever lifting argument.  

More recently, Davies, Jenssen, Perkins, and Roberts \cite{DJPR} gave an independent proof introducing an audacious new approach utilizing linear programming.  Following Davies et al., we will derive bounds on $P_G(\l)$ by considering the \emph{occupancy fraction}, denoted $\a_G(\l)$.  This is the expected fraction of vertices of $G$ belonging to a random independent set chosen according the hard-core model.  More explicitly, 
	\[
		\a_G(\l)=\frac{1}n \frac{\sum_{I\in \Ind(G)} \abs{I}\l^{\abs{I}}}{P_G(\l)}=\frac{1}n \frac{\l P'_G(\l)}{P_G(\l)}.
	\]
Davies et al. \cite{DJPR} proved the following.
\begin{theorem}[Davies, Jenssen, Perkins, Roberts]\label{thm:djpr}
	For all $d$-regular graphs $G$ and all $\l>0$, it is the case that
	\[
		\a_G(\l)\leq \a_{K_{d,d}}(\l).
	\]
\end{theorem}
Because $\a_G(\l)$ is essentially the logarithmic derivative of $P_G(\l)$, this is a strengthening of Theorem~\ref{thm:kz}.  Corollary~\ref{cor:main} below shows how to use the occupancy fraction to bound the independence polynomial.

In this paper, we investigate lower bound analogues of Theorem~\ref{thm:djpr}.  In Section~\ref{sec:regular}, we prove that for any $d$-regular graph, the occupancy fraction is bounded below by that of $K_{d+1}$.  This is a stronger form of the corresponding result for $\ind(G)$ proved by the authors in \cite{CR}.  In later sections, we discuss a problem raised by Kahn \cite{K15}, that of giving a lower bound on $P_G(\l)$ for $d$-regular triangle-free graphs.  We use the same occupancy fraction approach to give bounds in this case.

When $\l$ is large, the hard-core model is biased strongly towards large independent sets.  Indeed, 
\begin{equation}
	\lim_{\l\to\infty} \a_G(\l)=\frac{\a(G)}{n},\label{eqn:ir}
\end{equation} 
where $\a(G)$ is the \emph{independence number of $G$}, and the ratio $\a(G)/n$ is the \emph{independence ratio of $G$}.  Our methods yield a weak lower bound on the independence ratio of regular triangle-free graphs, which we present in Section~\ref{sec:dreg}.  In Section~\ref{sec:pete}, we focus on triangle-free cubic graphs.  Here we are able to give a bound that is relatively good when $\l=1$.  We conjecture that the Petersen graph is extremal, but this cannot be true for all $\l$ since the independence ratio for triangle-free cubic graphs is minimized by $\GP(7,2)$, a generalized Petersen graph (see Staton \cite{S}).


\section{Lower bounds for the hard-core model on regular graphs} 
\label{sec:regular}
	
The following lower bound on the occupancy fraction is the main theorem of this section.  Its proof appears at the end of this section after a number of lemmas concerning the linear programming approach.
	
\begin{theorem}\label{thm:main}
	If $G$ is a $d$-regular graph on $n$ vertices and $\l>0$, then 
	\[
		\a_G(\l)\geq \a_{K_{d+1}}(\l).
	\]
	As a consequence, we have
	\[
		P_G(\l)^{1/n}\geq P_{K_{d+1}}(\l)^{1/(d+1)}.
	\]
	Equality is, in both cases, only achieved for $G$ a disjoint union $K_{d+1}$s.
\end{theorem}
	
	Following Davies et al., we will consider, for each vertex in $V$, the probability that it belongs to, and the probability that it is covered by, a randomly chosen independent set.  To be explicit, if $I$ is an independent set, we say that $v\in V(G)$ is \emph{occupied} if $v\in I$ and \emph{uncovered} if $I\cap N(v)=\emptyset$.  If $I$ is distributed according to the hard-core model, we write $p_v$ for $\P(v\in I)$ and $q_v$ for $\P(\text{$v$ is uncovered})$.  Note that both $p_v$ and $q_v$ are functions of $\l$.  Also, it is the case that $p_v\leq q_v$ since $\setof{I\in \Ind(G)}{\text{$v$ is occupied}}\of \setof{I\in \Ind(G)}{\text{$v$ is uncovered}}$.
	
\begin{lemma}\label{lem:23}
	In the hard-core model on $G$ with fugacity $\l>0$, we have 
	\begin{enumerate}[1)]
		\item $p_v=\frac{\l}{1+\l}q_v$, and
		\item $\a_G(\l)=\frac{1}n\sum_{v\in V(G)} p_v$.
	\end{enumerate}
\end{lemma}
	
\begin{proof}
	For 1), note that the conditional probability that $v$ is occupied given that it is uncovered is $\l/(1+\l)$.  For the second part, simply write $\abs{I}=\sum_{v\in V(G)} \1(v\in I)$ and take expectations.  
\end{proof}
	
We will prove Theorem~\ref{thm:main} by computing the occupancy fraction in two different ways, each based on the neighborhood of a uniformly randomly chosen vertex $v$ in $G$.  In particular, we record the external influence of $I$ on $N(v)$.  

\begin{definition}
	Let $I\in \Ind(G)$ be chosen according to the hard-core model, and $v$ be chosen uniformly from $V(G)$, independently of $I$.  We define random variables 
	\[
		U=U(v,I)=N(v)\wo N(I\wo N(v)),\qquad\text{and}\qquad H=H(v,I)=G[U].
	\]
	Thus, $U$ is the subset of the neighborhood of $v$ which is not covered by any vertex of $I$ outside of $N(v)$.  
\end{definition}

We define this triple because, conditioning on $I\wo N(v)$, we have that $I\cap N(v)\of U$ and, moreover, $I\cap N(v)$ is distributed according to the hard-core model on $H$ with fugacity $\l$.

\begin{lemma}[Davies et al.~\cite{DJPR}]\label{lem:56}
	In the hard-core model on a $d$-regular graph $G$ with fugacity $\l>0$ we have, with the notation above,
	\begin{enumerate}[1)]
		\item $\displaystyle \a_G(\l)=\frac{\l}{1+\l}\E\left(\frac{1}{P_H(\l)}\right)$, and also
		\item $\displaystyle \a_G(\l)=\frac{\l}{d}\,\E\left(\frac{P'_H(\l)}{P_H(\l)}\right)$.
	\end{enumerate}
\end{lemma}

\begin{proof}
	Since $v$ is uniformly distributed on $V(G)$, Lemma~\ref{lem:23} yields
	\[
		\a_G(\l)=\E(p_v)=\frac{\l}{1+\l}\E(q_v)=\frac{\l}{1+\l}\E\left(\frac{1}{P_H(\l)}\right).
	\]
The final equality follows since $v$ is uncovered precisely if $I\cap N(v)=\emptyset$, and $I\cap N(v)$ is distributed according to the hard-core model on $H$.

	For the second part, we pick a random vertex of $V(G)$ by picking a uniformly random neighbor, say $u$, of $v$.  Since $G$ is regular, $u$ is also uniformly distributed on $V(G)$.  Thus,
	\[
		\a_G(\l)=\E\left(p_u\right)=\frac{1}{d}\,\E\left(\frac{\l P'_H(\l)}{P_H(\l)}\right),
	\]
	since $\E(\l P'_H(\l)/P_H(\l))$ is the expected number of occupied neighbors of $v$.
\end{proof}

We will now find the minimum value of $\E(1/P_H(\l))$, where the distribution of $H$ is no longer tied to that arising from some $d$-regular graph.  Instead, the distribution of $H$ will merely have to satisfy the very limited condition that the two expressions for $\a$ in Lemma~\ref{lem:56} agree.  We let $\H$ be the set of all graphs on at most $d$ vertices, including the graph with empty vertex set.  We say a random variable $H$ with values in $\H$ is \emph{neighborly} if  
\[
	\frac{\l}{1+\l}\E\left(\frac{1}{P_H(\l)}\right)=\frac{\l}{d}\,\E\left(\frac{P'_H(\l)}{P_H(\l)}\right).
\]
Now we define 
\[
	\a_*=\frac{\l}{1+\l}\inf\setof{\E\left(\frac{1}{P_H(\l)}\right)}{\text{$H$ is a neighborly probability distribution on $\H$}}.
\]
  This minimum is the optimal value of a linear program where the variables are the probabilities that the distribution of $H$ assigns to graphs in $\H$.  We write $p_H$ for these probabilities, and also set
\[
	a_H=\frac{1}{P_H(\l)}\qquad\text{and}\qquad b_H=\frac{(1+\l)P'_H(\l)}{dP_H(\l)}.
\]  
In standard form, the linear program is the following, which we refer to as $\LP(d,\l)$.
\begin{align*}
	\a_*=\min \frac{\l}{2(1+\l)} & \sum_{H\in \H} p_H(a_H+b_H)\quad\text{subject to}\\
	& \sum_{H\in \H} p_H=1,\\
	& \sum_{H\in \H} p_H(a_H-b_H)=0,\\
	& p_H\geq 0\quad\text{for all $H\in \H$}.
\end{align*}
	
	A number of times in this paper we use some basic facts about the solutions to linear programs and their duals.  Firstly, if we can find feasible solutions to a program and its dual with matching objective values, both solutions are optimal.  Secondly, there is a more subtle equivalent criterion for simultaneous optimality called \emph{complementary slackness}.  A pair of feasible solutions (one for the primal and one for the dual) satisfy complementary slackness if, for every matching pair of variable and constraint, either the constraint is tight or the variable is zero (or both).  See the text of Chv\'atal \cite{C} for more details.
	
	We will compute the solution to $\LP(d,\l)$ by exhibiting a solution to the primal and a solution to the dual which have matching objective values.  We will also verify uniqueness by using complementary slackness.  We start by writing down the dual problem.  It has two unbounded variables, $A$ and $B$, corresponding to the equality constraints in the original program and an inequality corresponding to each variable.
\begin{align*}
	\a_*&=\max \frac{\l}{2(1+\l)} A\quad\text{subject to}\\
	& A+B(a_H-b_H)\leq a_H+b_H\quad\text{for all $H\in \H$}.
\end{align*}
We will exhibit a dual feasible $(A,B)$ whose dual objective value agrees with the primal objective value for the distribution $(p_H)$ arising from taking $G=K_{d+1}$.  To that end, we prove a lemma that describes the $(A,B)$ that satisfy certain of the dual constraints with equality.

\begin{lemma}
	Suppose that $A=A_K,B=B_K\in \R$ are such that the dual constraints $A+B(a_H-b_H) \ge a_H+b_H$ are satisfied with equality for $H$ being the graph with no vertices (denoted $\0$) and also for $H=K\neq \0$. Then 
	\[
		A_K= \frac{2b_K}{1 - a_K + b_K} = \frac{2}{1 + \frac{\l d}{2(1+\l)}\frac{p'(K)}{\mu(K)}}\qquad\text{and}\qquad B_K=1-A_K,
	\] 
	where $p'(K)=\P(I\neq\0)$ and $\mu(K)=\E\abs{I}$ for $I$ distributed according to $\HC_K(\l)$. 
\end{lemma}

\begin{proof}
	Since $a_\0=1$ and $b_\0=0$, one of the equations that $A$ and $B$ satisfy is that $A+B=1$. Substituting into the constraint corresponding to $H=K$ gives 
	\[
		A_K = \frac{2b_K}{1 - a_K + b_K} = \frac{2b_K}{p'(K) + b_K} = \frac2{1 + \frac{p'(K)}{b_K}} = \frac{2}{1 + \frac{\l d}{2(1+\l)}\frac{p'(K)}{\mu(K)}}.\qedhere
	\]
\end{proof}

\begin{lemma}\label{lem:dualsat}
	With the notation of the previous lemma, if $H$ and $K$ are graphs on at least one and at most $d$ vertices, then the following are equivalent.
	\begin{enumerate}
		\item $A_K + B_K(a_H-b_H) \le a_H + b_H.$
		\item $A_K \le A_H.$
		\item $\displaystyle \frac{p'(K)}{\mu(K)} \ge \frac{p'(H)}{\mu(H)}.$
	\end{enumerate}
\end{lemma}

\begin{proof}
	For the first equivalence, note that the following are equivalent.
	\begin{align*}
		A_K + B_K(a_H-b_H) &\le a_H + b_H \\
		A_K + (1-A_K)(a_H - b_H) &\le a_H + b_H \\
		A_K (1 - a_H + b_H) &\le 2b_H \\
		A_K &\le A_H.
	\end{align*}
	The second equivalence comes from the fact that $A_K = 1/\bigl(1 + \frac{\l d}{2(1+\l)}\frac{p'(K)}{\mu(K)}\bigr)$ is a strictly decreasing function of $p'(K)/\mu(K)$.
\end{proof}

Now we claim that $A = A_{K_d} = 2(1+\l)/(1+(d+1)\l)$, $B=B_{K_d} = ((d-1)\l - 1)/(1+(d+1)\l)$ are dual feasible and give the same dual objective value as arises in the primal problem from taking $(p_C) = \bigl(p^{K_{d+1}}_C\bigr)$, the probability distribution arising from the graph $K_{d+1}$. Clearly this choice of $A,B$ gives 
\[
		\a = \frac{\l}{2(1+\l)} A_{K_d} = \frac{\l}{1+(d+1)\l} = \a(K_{d+1}).
\]
On the other hand all other dual constraints are satisfied.  By Lemma \ref{lem:dualsat}, for all $H$ with between $1$ and $d$ vertices, $A_K+B_K(a_H-b_H)\leq a_H+b_H$, since
\[
	\frac{p'(K_d)}{\mu(K_d)} = 1 \ge \frac{p'(H)}{\mu(H)}.
\]
Thus the minimum value of $\a$ is achieved for $G=K_{d+1}$.  To prove uniqueness here we need only observe that the last inequality is only tight when $H$ is also complete.  Thus, by complementary slackness, no distribution $(p_H)$ can be extremal unless it is supported on complete graphs $H$ (and the zero vertex graph). In particular no graph $G$ can be extremal unless $p_G$ is supported on these. The following lemma characterizes such graphs.

\begin{lemma}
	If $G$ is $d$-regular, and for all independent $I\of V(G)$ and all $v\in V$, we have $H=H(v,I)$ complete (or $\0$), then $G$ is a disjoint union of $K_{d+1}$s.
\end{lemma}

\begin{proof}
	Suppose $G$ is not a disjoint union of $K_{d+1}$s. Then there exists a vertex $v$ with non-adjacent neighbors $u,w$. Set $I=\set{u,w}$. Then $H[\set{u,w}] = E_2$ and in particular $H$ is neither $\0$ nor complete. 
\end{proof} 

\begin{proof}[Proof of Theorem~\ref{thm:main}]
	Given a $d$-regular graph $G$, consider the random variables $U$ and $H$ defined just before Lemma~\ref{lem:56}.  Clearly, $H$ is neighborly and thus 
	\[
		\a_G(\l)\geq \a_*=\a_{K_{d+1}}(\l).
	\]
	For the bound on the independence polynomial, note that
	\begin{align*}
		\log P_G(\l)&=\int_0^{\l} \frac{P'_G(t)}{P_G(t)} dt\\
		&= n\int_0^{\l} \frac{\a_G(t)}t dt\\
		&\geq n\int_0^{\l} \frac{\a_{K_{d+1}}(t)}{t} dt\\
		&= \frac{n}{d+1}\log P_{K_{d+1}}(\l).\qedhere
	\end{align*}
	
\end{proof}
	

\section{Minimizing the hard-core model for triangle-free $d$-regular graphs} 
\label{sec:dreg}

As Davies et al.~\cite{DJPR} noted, the linear programming approach is simpler when $G$ is triangle-free since the graph induced on $N(v)$, for any $v$, is empty.  Where before we had to define both $U$ and $H=G[U]$, now $U$, the set of uncovered neighbors of $v$, always induces an empty graph.  Thus, we need only keep track of how many neighbors of $v$ are uncovered.  Our approach will be to add further constraints to the linear program, thereby getting a better approximation to the actual minimum.  Pick an independent set $I$ according to the hard-core model and a vertex $v$ uniformly at random, independently of $I$.  Let $Y$ be the number of uncovered neighbors of $v$, so $\range(Y)=\set{0,1,2,\ldots,d}$.  The distribution of $Y$ is specified by the $d+1$ values $y_0,y_1,y_2,\ldots,y_d$, where $y_i=\P(Y=i)$.  

Restricting ourselves to the constraints analogous to those in $\LP(d,\l)$ unfortunately gives a rather weak bound on the hard-core model in triangle-free regular graphs.  In this section, we add a natural lower bound on $y_0$ and, in the next, we add a further constraint in the special case of $3$-regular graphs.  While neither result is best possible, we are able to give reasonably tight bounds.

The constraint we add in the general $d$-regular case is the simple one that $y_0\geq \a$.  This follows from the fact that any vertex in $I$ must necessarily have all of its neighbors uncovered.  Since, by Lemma~\ref{lem:56}, 
\[
	\a_G(\l)=\frac{\l}{d}\,\E\left(\frac{P'_H(\l)}{P_H(\l)}\right)\qquad\text{and}\qquad
	\frac{P'_{E_i}(\l)}{P_{E_i}(\l)}=\frac{i(1+\l)^{i-1}}{(1+\l)^i}=\frac{i}{1+\l},
\]
the constraint will be written
\[
	y_0-\sum_{i=1}^d \frac{i\l}{d(1+\l)} y_i\geq 0.
\]
So, the linear program, which we will denote $\LPP(d,\l)$, is
\[
\begin{array}{lrcl}
	\text{Minimize}  &\displaystyle \sum_{i=1}^d i\cdot y_i & & \\
	\text{subject to} & \displaystyle\sum_{i=0}^d y_i&=&1,\\
	& \displaystyle \sum_{i=0}^d \bigl(i-d(1+\l)^{-i}\bigr)y_i&=&0,\\
	& \displaystyle y_0-\sum_{i=1}^d \frac{i\l}{d(1+\l)}y_i&\geq& 0,\\
	& y_0, y_1, y_2, \ldots, y_d&\geq& 0.
\end{array}
\]
Note that the solution to $\LPP(d,\l)$ gives (up to a constant factor) a lower bound on $\a(G)$ for $G$ a triangle-free $d$-regular graph.

\begin{theorem}\label{thm:gensoln}
	For any $\l$, the solution to $\LPP(d,\l)$ is supported on three values.  In particular, if we define $0= m_{d-1}<m_{d-2}<\cdots<m_1$ by 
	\[
		m_i=\left(\frac{d}{i+1}\right)^{1/i}-1,
	\]
	then for $m_i\leq \l< m_{i-1}$ (where we let $m_0=\infty$), the only non-zero variables are $y_0$, $y_i$, and $y_{i+1}$, with optimal values
	\begin{align*}
		y^*_0&= \frac{\l(1+(i+1)\l)}{(1+\l)^{i+1}+\l(d+1+(d+i+1)\l)},\\
		y^*_i&= \frac{(1+\l)((i+1)(1+\l)^i-d)}{(1+\l)^{i+1}+\l(d+1+(d+i+1)\l)},\\
		y^*_{i+1}&= \frac{(1+\l)^2(d-i(1+\l)^{i-1})}{(1+\l)^{i+1}+\l(d+1+(d+i+1)\l)}.
	\end{align*}
\end{theorem}

\begin{proof}
	Note that the dual program is as follows.
	\[
	\begin{array}{lrcl}
		\text{Maximize} & S &&\\
		\text{subject to} & S-dM-A &\leq& 0,\\
		& \displaystyle S+\Bigl(i-\frac{d}{(1+\l)^i}\Bigr)M-\frac{i\l}{d(1+\l)}A&\leq& i,\quad\text{for each $1\leq i\leq d$,}\\
		& A&\geq& 0.
	\end{array}
	\]
	The solution to the dual program, for $m_i\leq \l\leq m_{i-1}$ (or in the case $i=1$, $m_1\leq \l$), is
	\begin{align*}
		S^*&=\frac{d(1+\l)(1+(i+1)\l)}{(1+\l)^{i+1}+\l(d+1+(d+i+1)\l)},\\
		M^*&=\frac{(1+\l)^{i+2}}{(1+\l)^{i+1}+\l(d+1+(d+i+1)\l)},\\
		A^*&=\frac{d(1+\l)[(1+\l)^{i+1}-(1+(i+1)\l)]}{(1+\l)^{i+1}+\l(d+1+(d+i+1)\l)}.\\
	\end{align*}
	To show that we have solved both linear programs, we simply need to show that both assignments are feasible and that they satisfy complementary slackness.  We obtained $S^*$, $M^*$, and $A^*$ by insisting that the constraints in the dual corresponding to primal variables $y_0$, $y_i$, and $y_{i+1}$ were tight.  Therefore, it only remains to verify that the $y^*_i$s are feasible and that $A^*\geq 0$.  The latter is clear since $1+(i+1)\l\leq (1+\l)^{i+1}$.  We obtained the $y^*_i$ by requiring all the constraints in the primal to be tight and all variables other than $y_0$, $y_i$, and $y_{i+1}$ to be zero.  We need to check that these three variables are non-negative, but this follows from the fact that $m_i\leq \l\leq m_{i-1}$.
\end{proof}

\begin{corollary}
	If $G$ is a $d$-regular triangle-free graph, then
	\[
		\a(G)\geq y^*_0= \frac{\l(1+(i+1)\l)}{(1+\l)^{i+1}+\l(d+1+(d+i+1)\l)},
	\]
	where $m_i\leq \l<m_{i-1}$.
\end{corollary}

\begin{proof}
	In our solution, the constraint stating $y_0\geq \a$ is tight.
\end{proof}

As a consequence of this corollary, we do get a weak lower bound on the independence ratio of regular triangle-free graphs.  As in (\ref{eqn:ir}), we have
\[
	\frac{\a(G)}n=\lim_{\l\to \infty} \a_G(\l)\geq \lim_{\l\to \infty} \frac{\l(1+2\l)}{(1+\l)^{2}+\l(d+1+(d+2)\l)}=\frac{2}{d+3}.
\]
The best known lower bound is due to Shearer \cite{Sh}, who proved the following.
\begin{theorem}[Shearer]
	If $G$ is a triangle-free $d$-regular graph, then $\a(G)/n(G)\geq f(d)$, where $f(d)$ is given by the recurrence
	\[
		f(0)=1, \qquad f(d)=\frac{1+(d^2-d)f(d-1)}{d^2+1}.
	\]
\end{theorem}
This bound is stronger than ours, but the bounds are reasonably close for small $d$.


\section{Minimizing the hard-core model for triangle-free cubic graphs} 
\label{sec:pete}

In this section, we focus on the special case of cubic graphs.  As in the previous section, we add the constraint $y_0\geq \a=\a_G(\l)$.  We also add a constraint that is a lower bound on $p_3$.  Let $T_3$ be the first three levels of the infinite $3$-regular tree (so that $T_3$ has ten vertices).  Also, recall that $N^2(v)=\setof{x\in V(G)}{d(x,v)=2}$ and $N^2[v]=\setof{x\in V(G)}{d(x,v)\leq 2}$.

\begin{lemma}
	If $G$ is a triangle-free cubic graph and $Y$ is the number of uncovered neighbors of a uniformly chosen vertex with respect to an independent $I$ chosen according to the hard-core model, then
	\[
		\P(Y=3)\geq \frac{(1+\l)^3}{P(T_3)}.
	\]
\end{lemma}

\begin{proof}
	We first note that
	\[
		\P(Y=3)=\P(N^2[v]\cap I\of N(v)).
	\]
	This follows from the fact that all of $v$'s neighbors are uncovered if and only if $v\not\in I$ and $N^2(v)\cap I=\emptyset$.  Let $A=N^2[v]\cap I$ and $W=V(G)\setminus N^2[v]$.  We will bound $\P(A\of N(v)\given I\cap W)$.  Since we've conditioned on $I\cap W$, we know that $A$ is distributed as the hard-core model on $G'=G[N^2[v]\setminus N(I\cap W)]$, i.e., the graph on $N^2[v]$ after deleting vertices with neighbors in $W$.  Hence, 
	\[
		\P(A\of N(v)\given I\cap W)=\frac{(1+\l)^3}{P(G')},
	\]
	since there are precisely eight possible values for $A$ and the generating function for their weights is $(1+\l)^3$.	It only remains to show that, for all $\l$, we have $P(G')\leq P(T_3)$.  Note that there is a size-preserving injection from $\Ind(G')$ to $\Ind(T_3)$ and the result follows.
\end{proof}

To specify the constraint $y_0\geq \a$, note that, using Lemma~\ref{lem:23}, we have
\[
	\a=\frac{1}n\sum_{v\in V(G)} p_v=\frac{1}n\cdot \frac{\l}{1+\l}\sum_{v\in V(G)} q_v =\frac{1}n\cdot \frac{\l}{1+\l}\sum_{v\in V(G)}\sum_{j=0}^3 \frac{y_j}{(1+\l)^{j}}=\frac{\l}{1+\l}\cdot \E[(1+\l)^{-Y}],
\]
where the third equality follows from the fact that $\P(\text{$v$ is uncovered}\given Y=j)=1/(1+\l)^j$.  Also, as observed by Davies et al., we can pick a uniformly random vertex of $G$ by first picking a uniformly vertex $v$ and then picking a uniformly random neighbor of $v$ since $G$ is regular.  Thus,
\[
	\E\,Y=\frac{1}n\sum_{v\in V(G)}\sum_{u\in N(v)} q_u=3\cdot\frac{1+\l}{\l}\a.
\]
So, we have the constraint $\E\,Y=3\E[(1+\l)^{-Y}]$.  Of course, $\sum_{i=0}^3 y_i=1$.

With our added constraints, the linear program becomes the following.  We write $\L$ for $\frac{(1+\l)^3}{P(T_3)}$.
\[
\begin{array}{lrcl}
	\text{Minimize} & y_1+2y_2+3y_3 & & \\
	\text{subject to} & \displaystyle\sum_{i=0}^3 y_i&=&1,\\
	& \displaystyle \sum_{i=0}^3 \bigl(i-3(1+\l)^{-i}\bigr)y_i&=&0,\\
	& \displaystyle y_0-\sum_{i=1}^3 \frac{i\l}{3(1+\l)}y_i&\geq& 0,\\
	& y_3&\geq& \displaystyle \L,\\
	& y_0, y_1, y_2, y_3&\geq& 0.
\end{array}
\]
Once again, our strategy will be to exhibit values for the $y_i$s together with values for the dual program which satisfy complementary slackness.

The dual program is
\[
\begin{array}{lrcl}
	\text{Maximize} & S-\L B &&\\
	\text{subject to} & S-3M-A &\leq& 0,\\
	& \displaystyle S+\Bigl(1-\frac{3}{1+\l}\Bigr)M+\frac{\l}{3(1+\l)}A&\leq& 1,\\
	& \displaystyle S+\biggl(2-\frac{3}{(1+\l)^2}\biggr)M+\frac{2\l}{3(1+\l)}A&\leq& 2,\\
	& \displaystyle S+\biggl(3-\frac{3}{(1+\l)^3}\biggr)M+\frac{3\l}{3(1+\l)}A-B&\leq& 3,\\
	& A,B&\geq& 0.
\end{array}
\]
The solution to the primal linear program is as follows.
\begin{equation}\label{eqn:primsoln}
\begin{aligned}
	y^*_0&=\frac{\l(1+2\l)}{\denom}+\frac{\l^3(1+\l)^2}{(\denom)P(T_3)},\\
	y^*_1&=\frac{-1+\l+2\l^2}{\denom}+\frac{(1+7\l+9\l^2+\l^3)(1+\l)^2}{(\denom)P(T_3)},\\
	y^*_2&=\frac{2(1+\l)^2}{\denom}-\frac{(2+14\l+21\l^2+8\l^3)(1+\l)^2}{(\denom)P(T_3)},\\
	y^*_3&=\frac{(1+\l)^3}{P(T_3)}.
\end{aligned}
\end{equation}
Each of the $y_i$s is non-negative for all $\l\geq 0$.  This obvious for $y_0$ and $y_3$; for $y_1$ and $y_2$ we verify that for each, written as a rational function with denominator $(\denom)P(T_3)$, the numerator have all nonnegative coefficients.  The optimal dual solution is
\begin{align*}
	S^*&=\frac{3(1+\l)(1+2\l)}{\denom},\\
	M^*&=\frac{(1+\l)^3}{\denom},\\
	A^*&=\frac{3\l^2(1+\l)}{\denom},\\
	B^*&=\frac{3\l^2}{\denom}.
\end{align*}

To check complementary slackness, we observe that in fact we have equality in all constraints in both linear programs.  Hence, complementary slackness yields that these are each optimal solutions to the corresponding programs.  This is summarized in the following theorem.

\begin{theorem}
	If $G$ is a triangle-free cubic graph on $n$ vertices, then the hard-core model on $G$ with fugacity $\l$ satisfies
	\[
		\a\geq y^*_0=\frac{\l(1+2\l)}{\denom}+\frac{\l^3(1+\l)^2}{(\denom)P(T_3)}.
	\]
\end{theorem}

\begin{proof}
	The solution $\a_*$ of the minimization problem above is attained for the solution (\ref{eqn:primsoln}).  Moreover, since one of our constraints is of the form $y_0\geq \a$, and this constraint is satisfied with equality, we have $\a_*=y^*_0$.  
\end{proof}

\begin{corollary}\label{cor:main}
	If $G$ is a triangle-free cubic graph on $n$ vertices, then for any $\l_0\geq 0$,
	\[
		\frac{1}n \log P_G(\l_0)\geq \int_0^{\l_0} \frac{y^*_0}{\l} d\l.
	\]
	In particular,
	\begin{equation}
		\ind(G)^{1/n}\geq 1.538339.\label{eqn:minind}
	\end{equation}
\end{corollary}

\begin{proof}
	We have, as in the proof of Theorem~\ref{thm:main},
	\[
		\log P_G(\l)\geq n\int_0^{\l} \frac{\a_*(t)}{t} dt= n\int_0^{\l} \frac{y^*_0}{t} dt.
	\]
	Numerical integration up to $\l=1$ gives the second inequality.
\end{proof}

Unfortunately, in contrast to the result of Davies et al.~and our result from Section~\ref{sec:regular}, we do not determine the extremal graph for the occupancy fraction.  It should be noted that, in a recent paper, Davies et al. \cite{DJPR2} give a lower bound on the independence number of triangle-free graphs of given maximum degree $d$ that is asymptotically correct as $d\to \infty$.  For $d=3$ and $\l=1$, their bound is
\[
	P_G(1)^{1/n}\geq \exp\biggl\{\frac{W(3\log 2)^2+2W(3\log 2)}{6}\biggr\}=1.516712\ldots,
\]
where $W$ is the Lambert $W$ function.

Two graphs provide some support that our bound is not far from being optimal.  For the Petersen graph, the left hand side of (\ref{eqn:minind}) is at most $1.54199$, whereas our bound gave $1.538339$.  One might even be tempted to think that the Petersen graph is the extremal graph for the occupancy fraction for all $\l$.  However, this cannot be true as a result of Staton \cite{S} yields that a related graph, the \emph{generalized Petersen graph $\GP(7,2)$} (see Figure~\ref{fig:gp}), has a smaller independence ratio and hence, by (\ref{eqn:ir}), has smaller occupancy fraction for large $\l$.

\begin{figure}[h]
\begin{center}
	\begin{tikzpicture}[every node/.style={circle, draw, fill=black,inner sep=0pt, minimum width=5pt},
						every path/.style={thick}]
		\foreach \a in {-13, 38, 90, 141, 192, 244, 295}
		{
			\draw (\a:2) node (O\a) {};
			\draw (\a:1) node (I\a) {};
		}
		\foreach \a/\b in {-13/295, 38/-13, 90/38, 141/90, 192/141, 244/192, 295/244}
		{
			\draw (O\a) -- (I\a);
			\draw (O\a) -- (O\b);		
		}
		\foreach \a/\b in {-13/244, 90/-13, 192/90, 295/192, 38/295, 141/38, 244/141}
			\draw (I\a) -- (I\b);
		
		
	\end{tikzpicture}
\end{center}
\caption{The generalized Petersen graph $\GP(7,2)$}\label{fig:gp}
\end{figure}
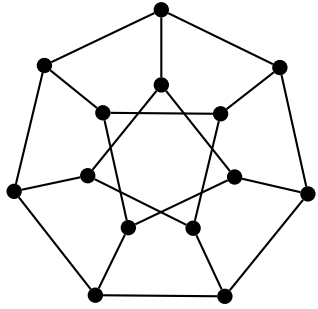

\begin{theorem}[Staton]
	If $G$ is a triangle-free cubic graph on $n$ vertices, then
	\[
		\frac{\a(G)}{n}\geq \frac{5}{14}=\frac{\a(\GP(7,2))}{n(\GP(7,2))}.
	\]
\end{theorem}

We therefore make the following conjecture.

\begin{conjecture}
	If $G$ is a triangle-free cubic graph on $n$ vertices, then
	\[
		P_G(\l)^{1/n}\geq \min\set{P_{\GP(5,2)}(\l)^{1/10},P_{\GP(7,2)}(\l)^{1/14}},
	\]
	where $\GP(5,2)$ is the Petersen graph.  In particular, 
	\[
		\ind(G)^{1/n}\geq \ind(\GP(5,2))^{1/10}=1.54198\ldots.
	\]
\end{conjecture}


\bibliographystyle{amsplain}
\bibliography{minhardcore}
	
\end{document}